\documentclass[12pt]{article}

\usepackage{amsfonts}

\usepackage{amscd}
\usepackage{amsthm}
\usepackage{indentfirst}

\usepackage{amssymb, amsmath, amsthm, amsgen, amstext, amsbsy, amsopn}
\usepackage{color}

\newtheorem{thm}{Theorem}[section]

\newtheorem{prop}[thm]{Proposition}
\newtheorem{cor}[thm]{Corollary}
\newtheorem{defin}[thm]{Definition}
\newtheorem{rem}[thm]{Remark}
\newtheorem{exam}[thm]{Example}

\newtheorem{question}[thm]{Question}

\newcommand{\R}{{\mathbb{R}}}

\newcommand{\T}{{\mathbb{T}}}
\newcommand{\Z}{{\mathbb{Z}}}
\newcommand{\N}{{\mathbb{N}}}
\newcommand{\C}{{\mathbb{C}}}
\newcommand{\SP}{{\mathbb{S}}}

\newcommand{\im}{{\it Im\, }}

\newcommand{\cC}{{\mathcal{C}}}

\newcommand{\cF}{{\mathcal{F}}}

\newcommand{\cH}{{\mathcal{H}}}
\newcommand{\cI}{{\mathcal{I}}}

\newcommand{\cL}{{\mathcal{L}}}

\newcommand{\hX}{{\widehat{X}}}
\newcommand{\hY}{{\widehat{Y}}}

\newcommand{\bkappa}{{\bar{\kappa}}}

\newcommand{\hpb}{\widehat{pb}}

\newcommand{\supp}{{\rm supp\ }}

\newcommand{\Qed}{\ \hfill \qedsymbol \bigskip}

\newcommand{\op}{{ \it Op}}

\DeclareMathOperator{\sgrad}{sgrad}

\begin{document}

\title{Lagrangian tetragons and instabilities in Hamiltonian dynamics}

\author{\textsc Michael Entov$^{a}$,\ Leonid
Polterovich$^{b}$ }

\footnotetext[1]{Partially supported by the Israel Science
Foundation grant $\#$ 1096/14 and by M. \& M. Bank Mathematics
Research Fund.}

\footnotetext[2]{ Partially supported by the Israel Science Foundation grant 178/13 and
the European Research Council Advanced grant 338809. }

\date{\today}

\maketitle

\begin{abstract}
We present a new existence mechanism, based on symplectic topology,
for orbits of Hamiltonian flows connecting a pair of disjoint subsets in the phase space.
The method involves function theory on symplectic manifolds combined with rigidity of Lagrangian
submanifolds. Applications include superconductivity channels in nearly
integrable systems and dynamics near a perturbed unstable equilibrium.

\end{abstract}

\tableofcontents

\section{Introduction and main results}
\label{sec-intro}

Given a Hamiltonian flow and a pair of disjoint subsets in the phase
space, does there exist an orbit connecting these subsets?  Various
instances of this question arise in the study of instabilities in
Hamiltonian dynamics. In this paper we combine two seemingly
remote approaches to detecting such orbits:
\begin{itemize}
\item the one based on the Poisson bracket invariants coming from function theory on symplectic manifolds \cite{BEP};

\item a geometric construction due to Mohnke \cite{Mohnke}, yielding what we call below a {\it Lagrangian tetragon},
which enabled him at the time to confirm a version of Arnold's chord
conjecture in Reeb dynamics.
\end{itemize}
The obtained package turns out to be efficient in a number of
specific models, including superconductivity channels in nearly
integrable systems and dynamics near a perturbed unstable
equilibrium that will be discussed below. The proposed existence
mechanism for connecting orbits is robust with respect to
perturbations of the Hamiltonian in the uniform norm.

\subsection{Interlinking: a chord existence mechanism}
\medskip
\noindent {\sc Setting the stage.}
Given an arbitrary smooth (possibly time-dependent) vector
field $v$ on a smooth manifold, the piece of its integral trajectory
defined over a time interval $[t_0,t_1]$, $t_0<t_1$, is called a
{\it chord of $v$}, or {\it a chord of the flow of $v$}, of {\it time-length} $t_1-t_0$. If such a chord passes
at the time $t_0$ through a set $X_0$ and at the time $t_1$ through
a set $X_1$ we say that it is {\it a chord from $X_0$ to $X_1$}.

If the flow of a vector field is defined everywhere for all times,
we say that the vector field is {\it complete}.

Let $(M,\omega)$ be a connected symplectic manifold. Consider a
smooth time-periodic Hamiltonian $G: M \times \SP^1 \to \R$, where
$\SP^1=\R/\Z$.
Given a pair of disjoint compact subsets $X_0,X_1 \subset M$, {\it a
(Hamiltonian) chord} of $G$ from $X_0$ to $X_1$ is a chord from
$X_0$ to $X_1$ of the Hamiltonian vector field defined by $G$. If
this vector field is complete, we say that $G$ is {\it complete}.

\medskip
\noindent{\sc Separation.} A function $G\in C^\infty (M\times \SP^1)$ {\it $\Delta$-separates} two
disjoint compact sets $Y_0,Y_1\subset M$ if
$$\Delta= \Delta(G;Y_0,Y_1): = \min_{Y_1\times \SP^1} G - \max_{Y_0\times \SP^1} G >0\;.$$
Note that in this definition the order of sets $Y_0$ and $Y_1$ is important: $G$ is larger on $Y_1$ than
on $Y_0$.

\medskip
\noindent{\sc Interlinking.}

\begin{defin}
\label{defin-interlinking} {\rm Let $(X_0,X_1)$, $(Y_0,Y_1)$ be two
pairs of disjoint sets: $X_0\cap X_1=Y_0\cap Y_1=\emptyset$.

We say
that the pair $(Y_0,Y_1)$ {\it $\kappa$-interlinks} the pair
$(X_0,X_1)$, $\kappa
>0$, if every complete (time-dependent) Hamiltonian which $\Delta$-separates $Y_0$
and $Y_1$ admits a chord from $X_0$ to $X_1$ of time-length $\leq
\kappa/\Delta$.

If this property is known to hold only for complete {\sl autonomous} Hamiltonians on $M$,
we say that the pair $(Y_0,Y_1)$ {\it autonomously $\kappa$-interlinks} the pair
$(X_0,X_1)$.
}
\end{defin}

Interlinking is the central phenomenon discussed in the present
paper. The existence of interlinking pairs is non-obvious and will
be discussed below. The interlinking of $(X_0,X_1)$ and $(Y_0,Y_1)$
can be formulated in terms of a Poisson bracket
invariant of the quadruple $(X_0,X_1,Y_0,Y_1)$ -- see
Section~\ref{sec-intro-pb4+}.

We will also discuss a stable version of the interlinking
phenomenon. Namely, given a closed connected manifold $K$, define the {\it
$K$-stabilization} of a subset $X \subset M$ as is its product
$X\times K\subset M\times T^* K$ with the zero-section $K$ of $T^*
K$ (the trivial case $K=T^* K= \{\textrm{a point}\}$ is also allowed;
thus, a set can be viewed as a trivial stabilization of itself).
We say that a pair $(Y_0,Y_1)$ {\it stably $\kappa$-interlinks}
the pair $(X_0,X_1)$ if for all $m\in\Z_{\geq 0}$ the pair
$(Y_0\times \T^m, Y_1\times \T^m)$ $\kappa$-interlinks the pair
$(X_0\times \T^m, X_1\times \T^m)$, where $\T^m$ is the $m$-dimensional torus.

\medskip
\noindent {\sc Robustness.} Observe that the existence of a
Hamiltonian chord of $G$ provided by the interlinking is robust with
respect to perturbations of $G$ for which the perturbed Hamiltonian
$G+F$ is complete and the perturbation term $F$ is sufficiently
small on $Y_0 \cup Y_1$. Indeed,
\begin{equation}\label{eq-separ-robust}
\Delta(G+F;Y_0,Y_1) \geq \Delta(G;Y_0,Y_1)-|\Delta(F;Y_0,Y_1)|\;.
\end{equation}
Therefore if $(Y_0,Y_1)$ $\kappa$-interlinks $(X_0,X_1)$ and $G$
$\gamma$-separates $X_0$ and $X_1$, the perturbed Hamiltonian $G+F$
has a chord from $X_0$ to $X_1$ of time-length $\leq
\kappa/(\gamma-\delta)$, where $\delta:=|\Delta(F;Y_0,Y_1)|$ is
assumed to be less than $\gamma$.

Let us emphasize that the perturbation $F$ can be arbitrarily large
away from $Y_0 \cup Y_1$ and can have large derivatives everywhere
-- thus the dynamics generated by $G+F$ might be completely
different from the dynamics generated by $G$ and the
chord of $G+F$ does not have to be close in any sense to the
chord of $G$.

\subsection{Introducing Lagrangian tetragons}
\label{subsec-intro-tetragons}
Next, we illustrate the notion of
interlinking for a special class of examples which plays a key role
in dynamical applications discussed further in the paper. The following construction
originates in the work of Mohnke \cite{Mohnke} (for another application of Mohnke's
construction to Hamiltonian dynamics see \cite{Lisi}).

Let $(\Sigma^{2k-1},\xi)$, $k\geq 1$, be a (not necessarily closed)
contact manifold with a co-orientable contact structure $\xi$ and
let $L$ be a closed connected Legendrian submanifold $L$ of
$\Sigma$. (If $k=1$, the contact structure $\xi$ is formed by the
zero subspaces of the tangent spaces of $\Sigma$ and the Legendrian
submanifold $L$ is just a point).

Let us make the following additional choices:

\bigskip
\noindent (C1) Pick a contact 1-form $\lambda_0$ on $\Sigma$,
$\xi=\ker \lambda_0$ (if $k=1$ we let $\lambda_0$ be any
non-vanishing 1-form on $\Sigma$). Denote by $\psi_t: \Sigma\to
\Sigma$ the Reeb flow of $\lambda_0$. (Recall that a contact form
$\lambda_0$ defines a vector field, called  Reeb vector field $v$,
by two conditions: $i_vd\lambda_0 =0$ and $\lambda_0 (v)=1$. The
flow of $v$ is called the Reeb flow of $\lambda_0$).

\bigskip
\noindent (C2) Pick any $T>0$ such that $\psi_t (L)$ is well-defined
and disjoint from $L$ for all $t\in (0,T]$ -- such a $T$ exists
since $L$ closed and tangent to the contact structure $\xi$ while
the Reeb vector  field $v$ generating the flow $\{ \psi_t\}$ is
nowhere tangent to $\xi$.

\bigskip
\noindent (C3) Let $0<R_0<R_1$.

\bigskip
\noindent (C4) Let $K$ be a closed connected manifold identified
with the zero-section of $T^* K$.

\bigskip
We will view the symplectization of $(\Sigma,\xi)$ as the symplectic
manifold $(\Sigma\times\R_+, d(s\lambda_0))$, where $s$ is the
coordinate on the $\R_+$-factor of $\Sigma\times\R_+$. Consider the
following quadruple of Lagrangian submanifolds (with boundary)
$\cF,\cC, \cL,\cH \subset \Sigma\times \R_+\times T^* K$:
\[
\cF:= \bigcup_{0\leq t\leq T} \psi_t (L)\times R_0 \times K,\;\;
\cC := \bigcup_{0\leq t\leq T} \psi_t (L)\times R_1 \times K,
\]
\[
\cL := \psi_T (L)\times [R_0,R_1]\times K,\ \
\cH := L\times [R_0,R_1] \times K.
\]

\begin{defin}
\label{defin-Lagr-tetragon} {\rm The union $\Lambda := \cF \cup \cC
\cup \cL \cup \cH $ is called a {\it Lagrangian tetragon}. The sets
$\cF$ and $\cC$ are called its {\it floor} and {\it ceiling} while
$\cL$ and $\cH$ the {\it low} and the {\it high} walls,
respectively.}
\end{defin}

We shall need the following glossary.

Assume that $U \subset \Sigma$ is a domain, $\cI \subset \R_+$ is an open interval containing $[R_0,R_1]$
and $W$ is an open tubular neighborhood of
$K$ in $T^* K$, so that $U\times\cI\times W$ contains the Lagrangian tetragon
$\Lambda$. The image of $\Lambda$ in an arbitrary symplectic
manifold $(M,\omega)$ under a symplectic embedding $U\times \cI\times W\to M$
will be called {\it a Lagrangian tetragon in $M$}. We will view it as a ``transplant" of the {\it original}
Lagrangian tetragon $\Lambda\subset \Sigma\times \R_+\times T^* K$ to
$(M,\omega)$. For the remainder of this section
we will use the same notation for an original Lagrangian tetragon and for its ``transplant".

A Lagrangian tetragon $\Lambda = \cF \cup \cC \cup \cL \cup \cH \subset (M,\omega)$ is
called {\it (stably/auto\-nomously) $\kappa$-interlinked (in $(M,\omega)$)} if the pair $(\cL,\cH)$
(stably/auto\-no\-mously) $\kappa$-interlinks the pair $(\cF,\cC)$ in $(M,\omega)$. Roughly speaking,
every complete (time-perio\-dic) Hamiltonian on $(M,\omega)$ which is larger on the high wall than on the low
wall admits a chord from the floor to the ceiling.

\medskip
\noindent
\begin{exam} [Prototype example]\label{exam-prototype}
{\rm Take $\Sigma$ to be the circle $\SP^1 = \R/\Z$ equipped with
the trivial $0$-dimensional contact structure determined by the
contact form $\lambda_0=du$, where $u$ is the coordinate on $\R$.
Fix a point $v \in \SP^1$ considered as a Legendrian submanifold of
$\Sigma$, along with some numbers $0 < R_0 < R_1$ and $0<T < 1$. Let
$K=\{\textrm{a point}\}$. The corresponding Lagrangian tetragon
$\Lambda$ in the symplectization of $\Sigma$ is a quadrilateral of
area $a= T(R_1-R_0)$ in the cylinder $\SP^1 (u) \times \R_+ (s)$,
with the area form $ds\wedge du$. It was shown in \cite{BEP} that
$\Lambda$ is stably $a$-interlinked. Moreover, this result is sharp:
$\Lambda$ is not $a'$-interlinked with any $a' < a$ (see
Remark~\ref{rem-the-ineq-for-pb4+-is-exact-for-cot-bundle-of-1-dim-mfd}).
}
\end{exam}

\medskip
\noindent This example can be extended to higher dimensions in two
different ways.

\medskip
\noindent
\begin{exam} \label{exam-scc-prep}
{\rm Let $T^*\T^k = \R^k \times \T^k$, $k>1$, be the cotangent bundle of the
torus  with the canonical coordinates $p \in \R^k$ and $q \in
\T^k=\R^k/\Z^k$ so that the symplectic structure is given by $dp
\wedge dq$. Let $|\cdot|$ denote the Euclidean norm. Consider the
contact manifold $\Sigma = \{|p|=1\} \subset T^*\T^k$
(called the space of co-oriented contact elements to the torus),
equipped with the contact form $\lambda_0 = pdq$.  The corresponding
Reeb flow is simply the Euclidean geodesic flow $(p,q) \to
(p,q+pt)$. The symplectization $\Sigma \times \R_+$ is identified symplectically
with $T^*\T^k \setminus \T^k$ by $(p,q,s) \mapsto (sp,q)$. Consider
the Lagrangian tetragon $\Lambda$ in $T^*\T^k$ associated to the Legendrian fiber
$\{|p|=1, q=0\}$ of $\Sigma$, some $0 < R_0 < R_1$,
$T>0$ and $K=\{\textrm{a point}\}$. It is well-defined provided $0< T<1/2$. Observe that its
floor, ceiling and the low wall lie in the boundary of the domain
$$E:= \{R_0 < |p| < R_1\} \times \{|q| < T\}\;,$$
while the high wall is the central fiber of the fibration $E \to \{|q| < T\}$.}
\end{exam}

\begin{thm}\label{thm-inter-1} $\Lambda$ is stably $\kappa$-interlinked with
$\kappa=T(R_1-R_0)$.
\end{thm}

\noindent
For the proof see Section~\ref{subsec-no-exact-Lagr-submfds-pfs}.

\medskip
\noindent
\begin{exam} \label{exam-ue-prep}
{\rm Consider the unit sphere $\SP^{2k-1}$ in the standard
symplectic space $\R^{2k}=\C^k$ equipped with the complex coordinates
$z=p+iq$ (in the vector notation) and the symplectic form $dp \wedge dq$. The sphere carries
a contact structure $\xi_{st}$ given by its field of complex tangencies which
is defined by the contact form $\lambda_0 = \frac{1}{2}(pdq-qdp)$. The
corresponding Reeb flow is given by $z \mapsto e^{2it}z$. The
symplectization $\Sigma \times \R_+$ is identified symplectically with
$\C^k\setminus 0$  by $(z,s) \mapsto \sqrt{s}z$.  Consider the
Lagrangian tetragon $\Lambda$ in $\C^k$ associated to the Legendrian sphere
$\{|p|=1, q=0\}$, some numbers $0 < R_0 < R_1$, $T=\pi/4$, and $K=\{\textrm{a point}\}$. Observe
that the high and the low walls of $\Lambda$ lie in spherical shells
$\{R_0^{1/2} \leq |p| \leq R_1^{1/2}, q=0\}$ and  $\{p=0,\;R_0^{1/2}
\leq |q| \leq R_1^{1/2}\}$ and  in the $p$- and in the $q$-space,
respectively. Its floor and ceiling lie in the
$(2k-1)$-dimensional spheres in $\C^k$ of radii $R_0^{1/2}$ and
$R_1^{1/2}$, respectively.}
\end{exam}

\begin{thm}\label{thm-inter-2}
$\Lambda$ is stably $\kappa$-interlinked with
$\kappa= \pi (R_1-R_0)/4$.
\end{thm}

\noindent
For the proof see Section~\ref{subsec-no-exact-Lagr-submfds-pfs}.

\subsection{Superconductivity channels}\label{subsec-scc}

Let us now present applications of the results above to
specific Hamiltonian systems.

The first application is related to an example that appears in the
famous work of Nekhoroshev \cite{Nekhoroshev-UMN} on the long-term
stability of nearly integrable systems\footnote{Our interest to this
model was triggered by a recent paper \cite{Bounemoura-Kaloshin} by
Bounemoura and Kaloshin.}.

Namely, let $\phi,I$ denote the standard action-angle coordinates on
the cotangent bundle $T^* \T^k = \T^k\times \R^k$ of a torus $\T^k$
(the $\phi$ coordinates are defined mod $1$ and the standard
symplectic form on $T^* \T^k$ is written as $dI\wedge d\phi$).
Nekhoroshev considered (analytic) Hamiltonians of the form
\begin{equation}
\label{eqn-nekhoroshev-hamiltonians} H(\phi,I) = h(I) + \epsilon
f(I, \phi),
\end{equation}
where $\epsilon$ is a small parameter, and studied the long-term
behavior $I(t)$ of the action variables $I$ along trajectories
of the flow generated by $H$. His main discovery was that if $h$ is
a so-called steep function (that is, its restriction to any
affine subspace of $\R^k$ has only isolated critical points), then, as long as
$\epsilon$ is sufficiently small, $I(t)$ stays close to $I(0)$ for
exponentially long times. At the same time Nekhoroshev gave an
explicit example of a Hamiltonian system of type
\eqref{eqn-nekhoroshev-hamiltonians} with two degrees of freedom
showing that if $h$ is not steep, then even for small $\epsilon$ the
actions $I(t)$ may grow linearly fast with $t$ (the so-called ``fast
diffusion" phenomenon) along certain chords of $H$ projecting to
straight intervals in a level set of $h$ in the $I$-coordinate space
(the so-called ``superconductivity channels").

Here we present a multi-dimensional Nekhoroshev-type example where
the linear growth phenomenon is robust with respect to
perturbations of the Hamiltonian that are $C^0$-small on a certain
``thin" subset of the phase space (but may be $C^1$-large
everywhere).

Namely, consider the time-independent Hamiltonian
\begin{equation}
\label{eq-H-Neh}
H(p,q,p',q') = \sum_{i=1}^k p_i h_i(p') +U(q)
\end{equation}
on the symplectic manifold
$$M:= T^*\T^k (p,q) \times T^*\T^m (p',q') \;.$$
Here the $p,p'$ and $q,q'$-coordinates correspond, respectively, to the
$I$ and $\phi$-coordinates above, $h_i$ are arbitrary smooth functions with $h_i(0)=0$, and $U$
is a non-constant potential. One easily checks that $H$ is complete.
If, for instance, all $h_i$'s are linear, $H$ is quadratic, albeit
non-convex, in momenta.

Looking at the Hamiltonian flow of $H$ one readily sees that if
$p'(0)=0$, then $q(t)=q(0)$ for all $t$ and $p(t) = p(0) - \nabla
U(q(0)) \cdot t$. Thus choosing $q(0)$ to be a non-critical point of
the potential, we see that the increment of the momentum
$|p(t)-p(0)|$ grows with linear speed along a straight line (a {\it
superconductivity channel}).

Assume now that the potential $U$ attains a local maximum $\beta:= U(0)$ at the point $0 \in \T^k$.
Take $r < 1/2$ and assume that
$\alpha:= \max_{\{|q|=r\}} U < \beta\;.$

Take the Lagrangian tetragon in $T^*\T^k$ as in
Example~\ref{exam-prototype}, if $k=1$, and as in
Example~\ref{exam-scc-prep}, if $k > 1$. Stabilizing it by the zero-section $\T^m
\subset T^*\T^m$, we get a Lagrangian tetragon in $M=T^* \T^k\times T^* \T^m$. Denote
by
$\cF$, $\cC$, $\cL$, $\cH$ its floor,
ceiling and walls. Observe that $H$ $\gamma$-separates the
walls $\cL$ and $\cH$ with $\gamma: = \beta-\alpha$. As an immediate
consequence of Theorem~\ref{thm-inter-1} (if $k>1$) and of the
statement at the end of Example~\ref{exam-prototype} (if $k=1$) we
get the following result.

\medskip
\begin{cor}\label{cor-diffu}
For every complete Hamiltonian $G=H+F$, where $F \in C^\infty(M
\times \SP^1)$ and $\delta:= |\Delta(F;\cL,\cH)| < \gamma$, there
exists a Hamiltonian chord of $G$ from $\cF$ to
$\cC$ of time-length $\leq (R_1 - R_0)r/(\gamma-\delta)$.
The increment of $|p|$ along this chord equals $R_1 -R_0$.
\end{cor}

\subsection{Unstable equilibrium}\label{subsec-ue}

The simplest model of an unstable equilibrium is provided by the
quadra\-tic Hamiltonian $H=\frac{1}{2}(|p|^2-|q|^2)$ on the standard
symplectic vector space $(\R^{2k},dp \wedge dq)$ (as before, $|\cdot
|$ denotes the Euclidean norm). Consider the Lagrangian tetragon
$\Lambda \subset \R^{2k}$ described in Example~\ref{exam-ue-prep}.
Denote by $\cF,\cC,\cL,\cH$ its floor, ceiling and walls. Take
any point of the form $z=u$, $u \in \R^k$, lying in the sphere
$\{|p|=R_0^{1/2}, q=0\} \subset \cH$. By definition, the point
$z'=e^{i\pi/4}z$ lies on the floor $\cF$. At the same time one
readily sees that $z'$ belongs to the unstable manifold $\{p=q\}$ of
the fixed point $0$ of the Hamiltonian flow of $H$. The trajectory
of $z'$ has the form $e^tz'$, and hence it eventually hits the
ceiling $\cC$. Thus we have found a chord of $H$ from $\cF$ to
$\cC$.

Observe also that $H$ $\gamma$-separates $\cL$ and $\cH$ with
$\gamma= R_0$. Therefore, by Theorem~\ref{thm-inter-2}, a chord
connecting $\cF$ and $\cC$ persists under sufficiently $C^0$-small
perturbations of $H$ on $\cL$ and $\cH$, yielding the following
corollary.

\medskip
\begin{cor}\label{cor-ue-1} For every complete Hamiltonian
$G=H+F$, where $F \in C^\infty(M \times \SP^1)$ and $\delta:=
|\Delta(F;\cL,\cH)| < R_0$, there is a Hamiltonian chord of $G$ from
$\cF$ to $\cC$ of time-length $\leq \pi(R_1-R_0)/(4(R_0-\delta))$.
The increment of $|(p,q)|$ along this chord equals
$R_1^{1/2}-R_0^{1/2}$.
\end{cor}

Here is another similar setting where Theorem~\ref{thm-inter-2} can be applied. Let $G:
\R^{2k} (p,q)\times \SP^1(t)\to \R$ be a complete Hamiltonian of the
form
\[
G(p,q,t) = |p|^2/2 + U(q,t).
\]
Assume that for some $0<R_0<R_1$
$$
\max_{\{R_0^{1/2} \leq |q| \leq R_1^{1/2}\}  \times \SP^1} U =: -\beta \leq 0
$$
and that $U (0,t)=0$ for all $t\in \SP^1$.
Then $G$ $\gamma$-separates $\cL$ and $\cH$ with
$\gamma = R_0/2 +\beta$. Thus, Theorem~\ref{thm-inter-2} yields the following corollary.

\begin{cor}\label{cor-mohnke-sphere-mechanical-hamilt}
There exists a Hamiltonian chord of $G$ from $\cF$ to $\cC$ of
time-length bounded from above by $\pi(R_1-R_0)/(2R_0+4\beta)$.
\end{cor}

This corollary can be viewed as a generalization of the
known result in the theory of the inverse Lagrange-Dirichlet problem
about the instability of the
equilibrium point of a mechanical system corresponding to a
(non-strict) local maximum of the potential (see \cite{Hagedorn,
Taliaferro}) -- such an instability follows immediately if the
potential $U$ above is taken to be time-independent and with a local
maximum $U(0)=0$ at $q=0$.

\subsection{A Poisson bracket invariant}
\label{sec-intro-pb4+}

It has been shown in \cite{BEP} that the existence of connecting
trajectories of Hamiltonian flows can be proved by means of a certain invariant involving
the Poisson bracket. In the present paper we refine this techniques
in order to establish interlinking for Lagrangian
tetragons.

Given a symplectic manifold $(M,\omega)$ and a quadruple of compact
sets $X_0,X_1,Y_0,Y_1\subset M$ with  $X_0\cap X_1 = Y_0\cap Y_1 =
\emptyset$, define a number
\[
pb^+_4 (X_0,X_1,Y_0,Y_1) := \inf \max_M\{F,G\}\;,
\]
where the infimum is taken over all compactly supported functions
$F,G: M\to\R$ such that $F|_{X_0}\leq 0$, $F|_{X_1} \geq 1$,
$G|_{Y_0}\leq 0$, $G|_{Y_1}\geq 1$. Let us explain the notation:
$pb$ stands for the ``Poisson bracket", $4$ for the number of sets
and $+$ for the fact that we are considering the maximum of the
Poisson bracket, instead of the uniform norm as it was done in
\cite{BEP}.

Write $\hX_i$, $\hY_i$ for the
$\SP^1$-stabilizations of $X_i$ and $Y_i$ in $M \times T^*\SP^1$.

Consider the set $\Upsilon$ (resp., $\Upsilon_{aut}$) of all $\kappa>0$ such that the pair $(Y_0,Y_1)$
$\kappa$-interlinks (resp., autonomously $\kappa$-interlinks) the pair $(X_0,X_1)$. Let
\[
\bkappa := \inf \Upsilon,\  \bkappa_{aut} :=\inf \Upsilon_{aut}.
\]
If $\Upsilon=\emptyset$ (resp., $\Upsilon_{aut}=\emptyset$),
set $\bkappa:=+\infty$ (resp., $\bkappa_{aut}:=+\infty$). One easily checks that $\bkappa_{aut}\leq \bkappa$ and $\Upsilon=[\bkappa,+\infty)\subset \Upsilon_{aut}=[\bkappa_{aut},+\infty)$.

\medskip
\noindent
\begin{thm}\label{thm-pb-new}
\[
1/pb_4^+ (X_0,X_1,Y_0,Y_1) = \bkappa_{aut} \leq \bkappa \leq 1/pb^+_4 (\hX_0,\hX_1,\hY_0, \hY_1).
\]
In particular, if
$pb^+_4 (\hX_0,\hX_1,\hY_0, \hY_1) =:1/\kappa >0$,
then the pair $(Y_0,Y_1)$
$\kappa$-interlinks the pair $(X_0,X_1)$,
and if $pb^+_4 (X_0,X_1,Y_0,Y_1) =:1/\kappa >0$,
then the pair $(Y_0,Y_1)$
autonomously $\kappa$-interlinks the pair $(X_0,X_1)$.

\end{thm}

\medskip
\noindent The proof follows the lines of \cite{BEP} with the
following amendments.

\medskip
\noindent 1. We adjust the argument of \cite{BEP} to complete
but not necessarily compactly supported Hamiltonians  appearing in
the definition of interlinking.

\medskip
\noindent
2. The results of \cite{BEP} do not say whether the
Hamiltonian chord connecting $X_0$ and $X_1$ goes from $X_0$ to
$X_1$ or from $X_1$ to $X_0$. This is why we introduce a refined version of the
Poisson bracket invariant and use a recent
theorem of A.Fathi (see Theorem~\ref{thm-fathi}) to detect
Hamiltonian chords going between two sets in a given direction.

\subsection{Interlinking and exact Lagran\-gians}
\label{subsec-interlink-tetragons-exact-Lagr}

\medskip

Let us discuss our method of proof of the interlinking of a Lagrangian tetra\-gon. For the sake of transparency, let us focus on the simplest case of autonomous interlinking of a Lagrangian tetragon $\Lambda=\cF\cup \cC\cup \cL\cup \cH$ in the symplectization
$\Sigma \times \R_+$ of a contact manifold $(\Sigma,\xi)$. By Theorem~\ref{thm-pb-new}, in order to establish the interlinking (respectively, the autonomous
interlinking) for a Lagrangian tetragon $\Lambda$, it suffices to show
the positivity of $\hpb^+_4 (\cF, \cC, \cL, \cH)$ (respectively, of
$pb^+_4 (\cF, \cC, \cL, \cH)$). We will prove this positivity as follows.

Assume that $\Lambda$ was constructed using a Legendrian submanifold $L\subset (\Sigma,\xi)$, a contact form $\lambda_0$ on $\Sigma$ and
real parameters $0< R_0 < R_1$ and $T>0$, see (C1)-(C3) in
Section~\ref{subsec-intro-tetragons} above\footnote{The manifold $K$ in (C4) is assumed to be the point.}.

The Lagrangian tetragon $\Lambda$ is a singular Lagrangian submanifold with corners.
One can smoothen its corners and get a smooth Lagrangian submanifold
$\Lambda_\varepsilon$ in $(\Sigma \times \R_+, d(s\lambda_0))$ diffeomorphic to $L\times
\SP^1$. One easily checks that the
Lagrangian isotopy class of $\Lambda_\varepsilon$ in $(\Sigma \times \R_+, d(s\lambda_0))$
depends only on the pair $(\Sigma, L)$ and
not on $\lambda_0$, $R_0$, $R_1$, $T>0$ and $\varepsilon$.

We will say that the pair $(\Sigma, L)$ is {\it exact}, if the above-mentioned
Lagrangian isotopy class of $\Lambda_\varepsilon$ in $(\Sigma \times \R_+, d(s\lambda_0))$
contains an
exact Lagrangian
submanifold -- that is, a Lagrangian submanifold $\Lambda'$ such that $(s\lambda_0)|_{\Lambda'}$ is an exact form. Otherwise $(\Sigma,L)$ will be called {\it non-exact}.

Similarly, we will say that the pair $(\Sigma, L)$ is {\it stably non-exact}, if for any $m\in\Z_{\geq 0}$
the Lagrangian isotopy class of $\Lambda_\varepsilon\times \T^m$ in the exact symplectic manifold
$(\Sigma\times \R_+\times T^* \T^m, d(s\lambda_0 + pdq))$ does not contain an exact Lagrangian submanifold.
(Here $pdq$ is the canonical 1-form on $T^* \T^m$). Clearly, stable non-exactness implies non-exactness.

\begin{thm}
\label{thm-non-exact-pairs-yield-interlinked-tetr-in-symplectization-1}
Let $(\Sigma,L)$ be a pair as above and let
$\Lambda = \cF\cup\cC\cup\cL\cup\cH\subset\Sigma\times\R_+$ be a
Lagrangian tetragon constructed using $L$, a contact form on $(\Sigma,\xi)$ and some parameters $R_0,R_1,T$.

\medskip
\noindent
A. If the pair $(\Sigma,L)$ is non-exact, then
$$pb_4^+ ( \cF, \cC, \cL, \cH)= \bigg( (R_1-R_0)T \bigg)^{-1}\;,$$
and thus, by Theorem~\ref{thm-pb-new},
$\Lambda$ is autonomously $(R_1-R_0)T$-interlinked in $(\Sigma \times \R_+, d(s\lambda_0))$.

\medskip
\noindent
B. If the pair $(\Sigma,L)$ is stably non-exact, then
$$
\hpb_4^+ ( \cF, \cC, \cL, \cH)= \bigg( (R_1-R_0)T \bigg)^{-1}\;,
$$
and thus, by Theorem~\ref{thm-pb-new},
$\Lambda$ is $(R_1-R_0)T$-interlinked in $(\Sigma \times \R_+, d(s\lambda_0))$.
\end{thm}

\medskip
\noindent We refer to Section~\ref{subsec-tetragons-main-result} for the proof and a more general version of this result. Theorem~\ref{thm-non-exact-pairs-yield-interlinked-tetr-in-symplectization-1}
gives rise to the following question.

\medskip
\noindent
\begin{question}\label{question-tetragon-interlinked-in-symplectization-itself}
Do exact pairs exist?
\end{question}

\medskip
\noindent We shall show, by using methods of symplectic topology,
that the answer is negative provided $(\Sigma, \xi)$ is the ideal contact boundary
 of a Liouville manifold (see
Section~\ref{subsec-no-exact-Lagr-submfds-pfs} below for the
statement of the result and Remarks~\ref{rem-weakly-exact-filling}
and \ref{rem-noncompact-Sigma} for
its generalization to certain non-compact contact manifolds). Note
that there do exist contact manifolds whose symplectizations contain
{\it some} exact closed Lagrangian submanifolds (see
\cite{Murphy-exact-Lagr}), but we do not know whether such a Lagrangian submanifold can be constructed as a smoothened
Lagrangian tetragon.

\medskip
\noindent
\begin{rem}\label{rem-Reeb}{\rm
Incidentally, Theorem~\ref{thm-non-exact-pairs-yield-interlinked-tetr-in-symplectization-1}
yields the following result in Reeb dynamics. As above, let $\psi_t: \Sigma\to\Sigma$ be the Reeb flow of the contact form $\lambda_0$ on $\Sigma$.
Let $(\Sigma,L)$ be a
non-exact pair and assume that,
as in the assumption (C2) in the construction of a Lagrangian tetragon,
$\psi_t (L)$ is well-defined and disjoint from $L$ for all $t \in (0,T]$ for some $T>0$. Put $L' := \psi_T(L)$. Let
$\lambda$ be an arbitrary contact form on $\Sigma$ that defines the same co-orientation of $\xi$ as $\lambda_0$ and has a complete Reeb flow.

We claim that the Reeb flow of $\lambda$ has a chord from $L$ to $L'$. Moreover,
the time-length of this chord does not
\textcolor{black}{exceed}\footnote{\textcolor{black}{The published version of the paper gives a wrong upper bound on the time-length. We correct this mistake here.}}\textcolor{black}{ $T\max_Y (\lambda/\lambda_0)$} with $Y := \bigcup_{t \in [0,T]} \psi_t(L)$.

Indeed, look at the Lagrangian tetragon $\Lambda = \cF \cup \cC \cup \cL \cup \cH$
associated to $L$, $\lambda_0$, $T$ and parameters $R=:R_1 >1=:R_0 >0$. Combining the anti-symmetry of the Poisson bracket invariant (see \eqref{eqn-pb4+-anti-symmetry} below) with
Theorem~\ref{thm-non-exact-pairs-yield-interlinked-tetr-in-symplectization-1} we get that
$$pb_4^+(\cH,\cL,\cF,\cC) = ((R-1)T)^{-1},$$ and hence the pair $(\cF,\cC)$ autonomously $\kappa$-interlinks
the pair of the walls $(\cH,\cL)$ in $(\Sigma \times \R_+, d(s\lambda_0))$ with  $\kappa:= (R-1)T$.
\textcolor{black}{The Reeb flow of $\lambda$ on $\Sigma$ lifts to a Hamiltonian flow on $(\Sigma\times \R_+, d(s\lambda_0))$ equivariant with respect to the multiplicative $\R_+$-action on $\Sigma\times \R_+$ and generated by the Hamiltonian $H:=s(\lambda_0/\lambda)$.}
Choose $R$ large enough and note that \textcolor{black}{$H$} $\Delta$-separates $\cF$ and $\cC$ with \textcolor{black}{ $\Delta:= R \min_Y (\lambda_0/\lambda) - \max_Y (\lambda_0/\lambda)$}. Hence, for such $R$,
the Reeb flow of $\lambda$ admits a chord of time-length $\leq \kappa/\Delta$ from
$L$ to $L'$. Letting $R \to +\infty$, we get \textcolor{black}{that the Reeb flow of $\lambda$ has a chord from $L$ to $L'$ of time-length $\leq T/\min_Y (\lambda_0/\lambda) =  T\max_Y (\lambda/\lambda_0)$, which proves the claim}. Let us mention that Legendrian contact homology
should be a more adequate and powerful technique for detecting Reeb chords connecting $L$ and $L'$, see e.g. \cite{Ekholm,Alves}.  }
\end{rem}

\medskip
\noindent
{\sc Organization of the paper:}
In Section~\ref{sec-prelim} we recall the necessary preliminaries. In Section~\ref{sec-existence-of-chords-for-smooth-vect-fields-Fathi} we formulate a recent result of A.Fathi on
the existence of chords for smooth vector fields which is crucial for our studies. In Section~\ref{sec-pb4-generalities} we discuss the Poisson bracket invariants and their relation to the existence of Hamiltonian chords. Section~\ref{sec-mohnke} is the central section of the paper -- in this section we prove the interlinking of
Lagrangian tetragons.

\section{Preliminaries}
\label{sec-prelim}

Let $(M^{2n},\omega)$ be a connected (not necessarily closed)
symplectic manifold.

Given an open set $U\subset M$, set $C^\infty_c (U)$ to be the space
of compactly supported smooth functions on $M$.

Further on we always identify $\R/\Z=\SP^1$. Given a Hamiltonian $G:
M\times \SP^1= M\times \R/\Z\to\R$, we denote $G_t := G(\cdot, t)$ and
say that $G$ is {\it compactly supported}, if $\supp G :=
\bigcup_{t\in \SP^1} \supp G_t$ is a compact subset of $M$. For a
bounded $G\in C^\infty (M)$ denote
\[
||G|| := \sup_M |G|.
\]

For $G\in C^\infty (M)$ define a vector field $\sgrad G$ by $i_{\sgrad G} \omega = -dG$.
Given $F,G\in C^\infty (M)$, define the Poisson bracket $\{F,G\}$ by
$$\{F,G\} := \omega(\sgrad G,\sgrad F) = dF (\sgrad G) = - dG (\sgrad F) =$$
$$= L_{\sgrad G} F = - L_{\sgrad F} G.$$

The {\it Hamiltonian flow}, or just a {\it flow}, of $G: M\times \SP^1\to\R$ is, by
definition, the flow of the (time-dependent) vector field $\sgrad
G_t$.

A subset of $(M,\omega)$ is called {\it displaceable}, if it can be completely displaced
from its closure by the flow of a compactly supported
(time-dependent) Hamiltonian.

We say that a (possibly open) symplectic manifold is of {\it bounded
geometry at infinity}, if it is geometrically bounded in the sense of
\cite{ALP} or convex at infinity in the sense of \cite{EG}. Let us
note that for any smooth manifold $N$ the cotangent bundle $T^* N$,
equipped with the standard symplectic structure, is of bounded
geometry at infinity.

A symplectic manifold $(M,\omega)$ is called {\it exact}, if $\omega$ is exact.
A Lagrangian submanifold $L$ of an exact symplectic manifold $(M,d\lambda)$
is called {\it exact} if the cohomology class $[\lambda|_L]\in H^1 (L;\R)$ is zero.

\section{Connecting trajecto\-ries of smooth vec\-tor fields}
\label{sec-existence-of-chords-for-smooth-vect-fields-Fathi}

In this section let $M$ be any smooth manifold, $v$ a complete smooth time-independent
vector field on $M$ and $X_0,X_1 \subset M$ disjoint compact subsets of $M$.
Denote by $T(X_0,X_1;v)$ the infimum of the time-lengths of the chords of $v$ from $X_0$ to $X_1$ (if there is no such chord, set $T(X_0,X_1;v):=+\infty$).
Define
\[
L_{max} (X_0,X_1;v) := \inf_F \max_M L_v F,
\]
where the infimum is taken over all smooth compactly supported functions $F$ on $M$ such that $F|_{X_0}\leq 0$, $F|_{X_1}\geq 1$.

A basic calculus argument shows that
\[
T(X_0,X_1;v) \geq 1/L_{max} (X_0,X_1;v).
\]

The following theorem has been proved by A.Fathi.

\begin{thm}[A.Fathi, \cite{Fathi}]
\label{thm-fathi}
$T(X_0,X_1;v) = 1/L_{max} (X_0,X_1;v)$.

In particular, together with the compactness of $X_0$ this implies that
if $L_{max} (X_0,X_1;v) >0$
there exists a chord of $v$ from $X_0$ to $X_1$
of time-length $1/L_{max} (X_0,X_1;v)$.
\end{thm}

\begin{rem}
\label{rem-fathi-thm} {\rm Replace the maximum of $L_v F$  by $|| L_v F ||$ in the definition of
$L_{max} (X_0,X_1;v)$  and call the resulting
quantity $L_0 (X_0,X_1;v)$. Then, as it was shown in \cite[Section
4.1]{BEP} by a rather basic averaging argument,
\begin{equation}
\label{eqn-L-0-time-length-of-chords-ineq}
\min \{ T(X_0,X_1;v), T(X_1,X_0;v)\} = 1/L_0 (X_0,X_1;v).
\end{equation}
In other words, if $L_0 (X_0,X_1;v) > 0$, then there exists {\it either}
 a chord of $v$ from $X_0$ to $X_1$ {\it or} a chord of $v$ from $X_1$ to $X_0$.
(Note that none of the results in \cite{BEP} about the existence of chords says anything about the direction of chords!)

Theorem~\ref{thm-fathi} is much more difficult than
\eqref{eqn-L-0-time-length-of-chords-ineq} -- in fact, Fathi's proof
of Theorem~\ref{thm-fathi} uses, along with the averaging argument
similar to the one in \cite[Section 4.1]{BEP}, some ingenious
arguments from the general theory of metric spaces. }
\end{rem}

\section{Poisson bracket invariants}
\label{sec-pb4-generalities}

Let $(M,\omega)$ be a connected symplectic manifold.

We say that sets $X_0$, $X_1$, $Y_0$, $Y_1\subset M$ form an {\it
admissible quadruple}, if they are compact and $X_0\cap X_1 = Y_0\cap Y_1 = \emptyset$.

Assume $U$ is an open subset of $M$ and $X_0, X_1, Y_0, Y_1\subset
U$ is an admissible quadruple. Define
\[
pb_4^U (X_0,X_1,Y_0,Y_1) := \inf_{F,G} ||\{F,G\}||,
\]
\[
pb_4^{U,+} (X_0,X_1,Y_0,Y_1) := \inf_{F,G} \max_M \{F,G\},
\]
\[
pb_4^{U,-} (X_0,X_1,Y_0,Y_1) := \inf_{F,G} (-\min_M \{F,G\}),
\]
where the infimum in all the cases is taken over all $F,G\in C^\infty_c
(U)$ such that
\begin{equation}
\label{eqn-F-G-X-Y-ineqs-defn-pb4} F|_{X_0}\leq 0, \ F|_{X_1} \geq
1, G|_{Y_0}\leq 0, G|_{Y_1}\geq 1.
\end{equation}
One can prove
similarly to \cite{BEP} that this class of pairs $(F,G)$ can be
replaced, without changing the infimums, by a smaller class where
the inequalities for $F,G$ are replaced by the equalities on some
open neighborhoods of the sets $X_0, X_1, Y_0, Y_1$.

Let us now define a stabilized version of $pb_4^{U,\pm}$.
Identify the cotangent bundle $T^* \SP^1$ with the cylinder $\R\times
\SP^1$ equipped with the coordinates $r$ and $\theta\ (\textrm{mod}\ 1)$ and the
standard symplectic structure $dr\wedge d\theta$.
As above, write $\hX_i$, $\hY_i$ for the
$\SP^1$-stabilizations of $X_i$ and $Y_i$ in $M \times T^*\SP^1$.
Given an open set $U\subset M$
and an admissible quadruple
$X_0, X_1, Y_0, Y_1\subset U$, set
\[
\hpb_4^{U,+} (X_0,X_1,Y_0,Y_1):= pb_4^{U\times T^* \SP^1, +}
(\hX_0, \hX_1, \hY_0, \hY_1),
\]
\[
\hpb_4^{U,-} (X_0,X_1,Y_0,Y_1):= pb_4^{U\times T^* \SP^1, -}
(\hX_0, \hX_1, \hY_0, \hY_1).
\]

\bigskip
\noindent
{\bf If $U=M$,
or it is clear from the context what $X_0,X_1,Y_0,Y_1$ are meant, we will omit the corresponding indices
and sets in the notation for $pb_4$, $pb^\pm_4$ and $\hpb_4$, $\hpb^\pm_4$.}

\bigskip

The following properties of $pb_4^\pm$, $\hpb_4^\pm$ follow easily
from the definitions, similarly to the corresponding properties for
$pb_4$ (cf. \cite{BEP}):

\medskip
\noindent{\sc Anti-symmetry:}
\begin{eqnarray}
\label{eqn-pb4+-anti-symmetry}
pb_4^{U,+} (X_0,X_1,Y_0,Y_1) = pb_4^{U,-} (X_1,X_0,Y_0,Y_1) = pb_4^{U,-} (X_0,X_1,Y_1,Y_0) =\\
=  pb_4^{U,-} (Y_0,Y_1,X_0,X_1) = pb_4^{U,+} (Y_1,Y_0,X_0,X_1)\nonumber,
\end{eqnarray}
\begin{equation}
\label{eqn-pb4-pb4+-comparison}
pb_4 \geq \max \{ pb_4^+, pb_4^-\}.
\end{equation}
Similar claims hold also for $\hpb_4^{U,\pm}$.

\medskip
\noindent{\sc Behavior under products:}

 Suppose that $M$ and $N$ are connected
symplectic manifolds. Equip $M\times N$ with the product symplectic
form. Let $K\subset N$ be a compact subset. Then for every
collection $X_0, X_1,Y_0,Y_1$ of compact subsets of $M$
\begin{equation}\label{eqn-pb-products}
pb_4^{M\times N,\pm} (X_0\times K, X_1\times K,Y_0\times K,Y_1\times
K)\leq pb_4^{M,\pm} (X_0, X_1,Y_0,Y_1)
\end{equation}
and a similar claim holds also for $\hpb_4^\pm$.

In particular, for any
$U\subset M$, $X_0,X_1,Y_0,Y_1\subset U$,
\[
\hpb_4^{U,\pm} (X_0,X_1,Y_0,Y_1)\leq pb_4^{U,\pm} (X_0,X_1,Y_0,Y_1).
\]

\medskip
\noindent{\sc Monotonicity:} Assume $U\subset W$ are opens sets in $M$ and $X'_0, X'_1, Y_0', Y'_1\subset U\subset W$ is an admissible quadruple. Let $X_0,X_1,Y_0,Y_1$ be another admissible quadruple such that
$X_0\subset X'_0, X_1\subset X'_1, Y_0\subset Y'_0, Y_1\subset Y'_1$. Then
\begin{eqnarray}
\label{eqn-pb4-monotonicity}
pb_4^{W,\pm} (X_0,X_1,Y_0,Y_1)\leq pb_4^{U,\pm} (X'_0,X'_1,Y'_0,Y'_1).
\end{eqnarray}
A similar inequality holds also for $\hpb_4^\pm$.

\medskip
\noindent{\sc Semi-continuity:}

Suppose that a sequence
$X_0^{(j)}, X_1^{(j)}, Y_0^{(j)}, Y_1^{(j)}$, $j\in\N$, of ordered collections
converges (in the sense of the Hausdorff distance between sets) to a collection $X_0,X_1,Y_0,Y_1$. Then
\begin{equation}
\label{eqn-Hausdorff-convergence}
\limsup_{j\to +\infty} pb_4^\pm (X_0^{(j)}, X_1^{(j)}, Y_0^{(j)}, Y_1^{(j)}) \leq pb_4^\pm (X_0,X_1,Y_0,Y_1).
\end{equation}
A similar inequality holds also for $\hpb_4^\pm$.

\bigskip
\noindent
{\bf Proof of Theorem~\ref{thm-pb-new}.}
Let us prove that
\begin{equation}
\label{eqn-1-over-pb4+=bkappa-aut}
1/pb_4^+ (X_0,X_1,Y_0,Y_1)= \bkappa_{aut}.
\end{equation}
To prove that
\begin{equation}
\label{eqn-1-over-pb4+-leq-1-bkappa-aut}
1/pb_4^+ (X_0,X_1,Y_0,Y_1)\leq \bkappa_{aut}.
\end{equation}
we need to prove that
\begin{equation}
\label{eqn-pb4+-geq-1-over-kappa-aut}
pb_4^+ (X_0,X_1,Y_0,Y_1)\geq 1/\kappa
\end{equation}
for any $\kappa \in\Upsilon_{aut}$. Pick such a $\kappa\in\Upsilon_{aut}$ and any $F,G\in C^\infty_c
(M)$ satisfying \eqref{eqn-F-G-X-Y-ineqs-defn-pb4}.
Then the pair $(Y_0,Y_1)$
autonomously $\kappa$-interlinks the pair $(X_0,X_1)$ and, since $G$ 1-separates $Y_0$ and $Y_1$, this means that there exists a chord of $G$
from $X_0$ to $X_1$ of time-length $\leq \kappa$. Restricting $F$ to the chord and applying the mean value theorem
from the basic calculus one readily obtains that
the function $L_{\sgrad G} F=\{ F,G\}$ takes a value greater or equal to $1/\kappa$ at some point of the chord
and thus $\max_M \{ F,G\}\geq 1/\kappa$. Since this holds for any $F,G\in C^\infty_c (M)$
satisfying \eqref{eqn-F-G-X-Y-ineqs-defn-pb4}, we obtain \eqref{eqn-pb4+-geq-1-over-kappa-aut} and hence
\eqref{eqn-1-over-pb4+-leq-1-bkappa-aut}.

Let us prove that
\begin{equation}
\label{eqn-1-over-pb4+-geq-1-bkappa-aut}
k_{aut}:=1/pb_4^+ (X_0,X_1,Y_0,Y_1)\leq \bkappa_{aut}.
\end{equation}
Equivalently, this means to prove the following: Let $G : M\to\R$ be a complete Hamiltonian that
$\Delta$-separates $Y_0$ and $Y_1$. We need to show that there exists a chord of $G$ from $X_0$ to
$X_1$ of time-length $\leq \kappa_{aut}/\Delta$.

We may assume without loss of generality that $\max_{Y_0} G =0$, $\min_{Y_1} G =1$, $\Delta=1$ (the general case is reduced to
this one if one replaces $G$ with $u\circ G$ for an appropriate function $u:\R\to\R$). Let $g_t$ be the flow of $G$. Since the flow $g_t$ is defined for all $t\in\R$, the following subset of $M$ is well-defined and compact:
\[
\Theta_{aut} := \bigcup_{0\leq t\leq \kappa_{aut}} g_t (X_0) \cup Y_0\cup Y_1.
\]
Let $G': M\to\R$ be a compactly supported Hamiltonian which coincides with
$G$ on an open neighborhood of $\Theta_{aut}$. Since $G$ coincides with $G'$ on
$Y_0$ and $Y_1$, we get that
\[
\max_{Y_0} G = \max_{Y_0} G' = 0,\ \min_{Y_1} G = \min_{Y_1} G' = 1.
\]
Since
\[
\kappa_{aut}=1/pb^+_4 (X_0,X_1,Y_0,Y_1) >0,
\]
we get that
\[
\max_M \{ F',G'\} = \max_M L_{\sgrad G'} F' \geq
\]
\[
\geq L_{max} (X_0,X_1; \sgrad G')
\geq pb^+_4 (X_0,X_1,Y_0,Y_1) = 1/\kappa_{aut}
\]
for any compactly supported $F': M\to\R$ such that $F'|_{X_0}\leq 0$, $F'|_{Y_1}\geq 1$.
Therefore, by Theorem~\ref{thm-fathi}, there exists a chord of $G'$ from $X_0$ to $X_1$
of time-length $\leq \kappa_{aut}$. This chord of $G'$ is also a chord of $G$ -- indeed, $G'$ coincides with $G$ on a neighborhood of $\Theta_{aut}$ and therefore for any $t\in [0,\kappa_{aut}]$ the time-$[0,t]$ flows of $G$ and $G'$ coincide
on $X_0$.

Thus we have proved the existence of a chord of $G$ from $X_0$ to $X_1$
of time-length $\leq \kappa_{aut}$. This finishes the proof of \eqref{eqn-1-over-pb4+-geq-1-bkappa-aut}. Together
\eqref{eqn-1-over-pb4+-leq-1-bkappa-aut} and \eqref{eqn-1-over-pb4+-geq-1-bkappa-aut} imply \eqref{eqn-1-over-pb4+=bkappa-aut}.

Let us now prove that
\begin{equation}
\label{eqn-kappa-leq-1-over-hpb4+}
\bkappa \leq \kappa := 1/pb^+_4 (\hX_0,\hX_1,\hY_0, \hY_1).
\end{equation}
Equivalently, this means to prove the following: Let $G : M\times \SP^1\to\R$ be a complete Hamiltonian that
$\Delta$-separates $Y_0$ and $Y_1$. We need to show that there exists a chord of $G$ from $X_0$ to
$X_1$ of time-length $\leq \kappa/\Delta$.

The argument is very similar to the argument used above in the autonomous case.
Namely,
consider an {\it autonomous} Hamiltonian
$$H: M \times T^* \SP^1 \to \R,\ (x,r,\theta) \to G(x,\theta)+r,$$
generating a Hamiltonian flow $h_t: M \times T^* \SP^1\to M \times T^* \SP^1$.
Note that the projection $M\times T^* \SP^1\to M$ maps each trajectory of $h_t$
to a trajectory of $g_t$ of the same time-length.
Thus, it suffices to prove the existence of a chord of $H$ from $\hX_0$ to $\hX_1$
of time-length $\leq \kappa/\Delta$.

The Hamiltonian flow $h_t$ of $H$ is defined for all times, since so is the Hamiltonian flow $g_t$ of $G$.
Note that $r=0$ on $\hY_0$, $\hY_1$. Therefore
\[
\max_{\hY_0} H = \max_{Y_0\times \SP^1} G,\ \min_{\hY_1} H = \min_{Y_1\times \SP^1} G,
\]
meaning that $H$ $\Delta$-separates $\hY_0$ and $\hY_1$.
Thus, $H$ is a complete autonomous Hamiltonian on $M\times T^* \SP^1$ that $\Delta$-separates $\hY_0$ and $\hY_1$.
By the result in the autonomous case, $H$ has a chord from $\hX_0$ to $\hX_1$ of time-length $\leq \kappa/\Delta$,
as required.
This finishes the proof of \eqref{eqn-kappa-leq-1-over-hpb4+} and of
the theorem.\Qed

\section{Lagrangian tetragons}
\label{sec-mohnke}

The proof of the positivity of $\hpb_4$ for Lagrangian tetragons
relies on the following
proposition which is based on a method from \cite{BEP} relating
deformations of symplectic forms and Poisson brackets.

\begin{prop}
\label{prop-FdG-Lambda}
Let $(M, \omega=d\lambda)$ be a connected (possibly open) exact symplectic manifold with a fixed primitive $\lambda$
and let $\Lambda\subset M$ be a closed Lagrangian submanifold.
Let
$F,G\in C^\infty_c (M)$.
Assume

\medskip
\noindent
1. $(M,\omega)$ does not admit
exact closed Lagrangian submanifolds Lagrangian isotopic to $\Lambda$.

\medskip
\noindent 2. The form $FdG$ is closed on $\Lambda$.

\medskip
\noindent 3. There exists a number $c >0$ so that the cohomology class of $FdG$ on $\Lambda$ satisfies
$[FdG|_\Lambda] = -c[\lambda|_\Lambda]$.

\medskip
\noindent
Then $\max_M \{ F,G\}\geq c$.

\end{prop}

\bigskip
\noindent
{\bf Proof of Proposition~\ref{prop-FdG-Lambda}.}
Following \cite{BEP}, consider the deformation
\[
\omega_\tau := \omega + \tau dF\wedge dG, \ \tau\in\R.
\]
A direct calculation shows that
\[
dF\wedge dG\wedge \omega^{n-1} = -\frac{1}{n}\{ F,G\} \omega^n.
\]
and thus
\[
\omega_\tau^n = (1-\tau \{ F,G\})\omega^n.
\]
Thus $\omega_\tau$ is symplectic for any $\tau\in I:=[0, 1/\max_M \{ F,G\} )$.

For any $\tau\in I$ the form $\omega$ can be mapped
(using Moser's method \cite{Moser}) to $\omega_\tau$ by a compactly supported
isotopy $\vartheta_\tau: (M, \omega_\tau)\to (M, \omega)$. Note that, by condition 2,
$\Lambda$ is Lagrangian with respect to $\omega_\tau$.
Therefore
the manifold
$\Lambda_\tau  :=
\vartheta_\tau (\Lambda)$ is a Lagrangian submanifold of
$(M,\omega)$ Lagrangian isotopic to $\Lambda$.

Assume, by contradiction, that $\max_M \{ F,G\}< c$, i.e. $a:= 1/c \in I$.
As we noted above $\Lambda_{a}$ is a Lagrangian submanifold with respect
to $\omega$ Lagrangian isotopic to $L$. We claim that $\Lambda_a$, and hence $\Lambda$,
is Lagrangian isotopic to an exact Lagrangian submanifold of
$(M,\omega)$ --  this would yield a contradiction with condition 1 and
thus prove that, in fact, $\max_M \{ F,G\}\geq c$.

Indeed, by condition 3, $\Lambda$ is exact with
respect to the primitive $\lambda_a:= \lambda + aFdG$ of $\omega_a$. Therefore
$\Lambda_a = \vartheta_a (\Lambda)$ is exact with respect to the primitive form
$\alpha:=(\vartheta_a^*)^{-1} \lambda_a$ of $d\alpha=\omega$.
Since the isotopy
$\vartheta_\tau$ and the functions $F,G$ are compactly supported, we get that
$\beta:= \alpha-\lambda$
is a closed compactly supported 1-form on $M$.

Let $f_t : M \to M$ be the locally Hamiltonian
flow of $\beta$ generated by
the vector field $v$ with $i_v \omega=\beta$. Since $\beta$ is compactly supported, $f_t$ is defined for all $t$. Also, since $f_t$ is a locally Hamiltonian flow, each $f_t$ is a symplectomorphism of $(M,\omega)$.
Observe that for all $t$
\[
\frac{d}{dt} f_t^* \lambda = f_t^* L_v \lambda = d (f_t^* (i_v \lambda)) + f_t^* \beta
\]
and
\[
f_t^* \beta = \beta.
\]
Therefore for all $t$
\[
\frac{d}{dt} [f_t^* \lambda ] = [\beta],
\]
which yields
\[
[f_1^* \lambda|_{\Lambda_a}] =
[\lambda|_{\Lambda_a}] + [\beta|_{\Lambda_a}] = [\alpha|_{\Lambda_a}] = 0.
\]
Thus $f_1 (\Lambda_a)$ is an
exact Lagrangian submanifold of $(M,\omega)$ with respect to $\lambda$, and $f_t (\Lambda_a)$, $t\in [0, 1]$, is the desired Lagrangian isotopy connecting $\Lambda_a = f_0 (\Lambda_a)$ to $f_1 (\Lambda_a)$. This finishes the proof of the claim and of the proposition.
\Qed

\subsection{Construction and interlinking of Lag\-ran\-gian tetra\-gons}
\label{subsec-tetragons-main-result}

Let $(\Sigma^{2k-1},\xi)$, $k\geq 1$, be a connected (not
necessarily closed) contact manifold with a co-orientable contact
structure $\xi$ (if $k=1$, the contact structure $\xi$ is formed by the zero subspaces of the tangent spaces).
Let $L$ be a closed connected Legendrian submanifold of $\Sigma$ (in the case $k=1$ the submanifold $L$ is just a point).

Pick a contact form $\lambda_0$ on $\Sigma$: $\xi = \{ \lambda_0=0\}$. (In the case $k=1$ let $\lambda_0$ be any non-vanishing form on $\Sigma$). Denote by $\psi_t:
\Sigma\to \Sigma$ the Reeb flow of $\lambda_0$.
(In the case $k=1$ the Reeb vector field $v$
is defined just by the condition $\lambda_0 (v)\equiv 1$).
Pick $T>0$ so that $\psi_t (L)$ is well-defined and disjoint from $L$ for all
$t\in (0,T]$.
Let $0<R_0<R_1$.

Let
$(\Sigma\times\R_+,
d(s\lambda_0))$ be the symplectization of $(\Sigma,\xi)$ (here $s$ is the coordinate
on the $\R_+$-factor of $\Sigma\times\R_+$).

Define a {\it Lagrangian tetragon} $\Lambda''\subset \Sigma\times\R_+$:
\[
\cF'':= \bigcup_{0\leq t\leq T} \psi_t (L)\times R_0,\ \
\cC'':= \bigcup_{0\leq t\leq T} \psi_t (L)\times R_1,
\]
\[
\cL'' := \psi_T (L)\times [R_0,R_1],\ \
\cH'' := L\times [R_0,R_1],
\]
\[
\Lambda'':= \cF''\cup\cC''\cup\cL''\cup\cH''.
\]
This is a singular Lagrangian submanifold. We will define its smoothening as follows.

Consider an embedding $\Phi: L\times (\R_+\times [0,T])\to
\Sigma\times \R_+$ given by
\[
\Phi (x \times (s,t)) := (\psi_t (x),s).
\]
Then for any smooth embedded loop $\gamma\subset [R_0,R_1]\times
[0,T]$ the restriction of $\Phi$ to $L\times \gamma$ is a Lagrangian
embedding. Choose a family $\gamma_\varepsilon$ of smooth embedded
loops in $[R_0,R_1]\times [0,T]$ $C^0$-converging to the boundary of
the rectangle $[R_0,R_1]\times [0,T]$ as $\varepsilon\to 0$ and
denote by
\[
\Lambda''_\varepsilon := \Phi (L\times \gamma_\varepsilon)\subset
\Sigma\times \R_+
\]
the resulting smooth Lagrangian submanifold, called a {\it smoothened Lag\-ran\-gian tetragon} in $\Sigma\times\R_+$.

Let $K$ be a closed connected manifold.
Then
\[
\Lambda':= \Lambda''\times K
\]
is a {\it Lagrangian tetragon in} $\Sigma\times\R_+\times T^* K$
and $\Lambda'_\varepsilon:=\Lambda''_\varepsilon\times K$ is its smoothening.

Assume now that
$U$ is a domain in $\Sigma$, $\cI\subset \R_+$ is an open interval containing $[R_0,R_1]$
and $W$ is an open tubular neighborhood of $K$ in $T^* K$ so that $\Lambda'$ is contained in $U\times\cI\times W$.
Denote by $\eta$ the standard 1-form on $T^* K$.
Let $(M,\omega=d\lambda)$ be an exact connected symplectic manifold and let $\nu: U\times \cI\times W\to M$ be a symplectic embedding such that
\begin{equation}
\label{eqn-nu-respects-primitives}
\nu^* \lambda = s\lambda_0\oplus \eta.
\end{equation}
Denote by $\Lambda := \nu (\Lambda')$ the resulting {\it Lagrangian tetragon in $M$}
and let
\[
\cF := \nu (\cF''\times K),\ \cC:= \nu (\cC''\times K),\ \cL:= \nu (\cL''\times K),\ \cH:= \nu (\cH''\times K),
\]
be, respectively, its floor, ceiling, low wall and high wall.
Let $\Lambda_\varepsilon := \nu (\Lambda'_\varepsilon)$ be its smoothening.

Note that for
sufficiently small $\varepsilon$ all Lagrangian
submanifolds $\Lambda_\varepsilon$ are Lagrangian isotopic in
$(M,\omega)$. From this point on we consider only such
$\varepsilon$.

\begin{thm}
\label{thm-pb4-Mohnke}
Assume that
$M$ does not admit closed
exact
Lagrangian submanifolds Lagrangian isotopic to
$\Lambda_\varepsilon$.
Then
\[
pb_4^{M,+} (\cF, \cC, \cL, \cH)=
\bigg( (R_1-R_0)T \bigg)^{-1}  >0,
\]
and thus,
$\Lambda$ is autonomously $(R_1-R_0)T$-interlinked.
\end{thm}

\medskip
\noindent
The autonomous interlinking of $\Lambda$ stated in the theorem
follows, by Theorem~\ref{thm-pb-new}, from the lower bound on $pb_4^{M,+} (\cF, \cC, \cL, \cH)$.
An immediate consequence of Theorem~\ref{thm-pb4-Mohnke} is as follows:

\medskip
\noindent
\begin{cor}
\label{cor-main} Assume that for every $m \in \N$,
$M \times T^* \T^m$ does not admit closed
exact
Lagrangian submanifolds Lagrangian isotopic to
$\Lambda_\varepsilon \times \T^m$.
Then
$\Lambda$ is stably $(R_1-R_0)T$-interlinked.
\end{cor}

\medskip
\noindent
{\bf Proof of Theorem~\ref{thm-non-exact-pairs-yield-interlinked-tetr-in-symplectization-1}:}
Part A immediately
follows from Theorem~\ref{thm-pb4-Mohnke} with $K=\{\textit{point}\}$, $M=\Sigma\times\R_+$ and $\nu=Id$, and, similarly,
part B follows from Corollary \ref{cor-main}.
\Qed

\medskip
\noindent {\bf Proof of Theorem~\ref{thm-pb4-Mohnke}.}
1) First, let us prove
\begin{equation}
\label{eqn-upper-bound-on-pb4+-of-a-tetragon}
pb_4^{M,+} ( \cF, \cC, \cL, \cH)\leq \bigg( (R_1-R_0)T \bigg)^{-1}.
\end{equation}
By the symplectic invariance, the monotonicity (see \eqref{eqn-pb4-monotonicity}) and the product (see \eqref{eqn-pb-products}) properties of $pb_4^+$,
it suffices to prove that
\begin{equation}
\label{eqn-upper-bound-on-pb4+-of-a-tetragon-in-sympln}
pb_4^{\Sigma\times\cI,+} ( \cF'', \cC'', \cL'', \cH'')\leq \bigg( (R_1-R_0)T \bigg)^{-1}.
\end{equation}
Let us consider a Hamiltonian of the form $u(s)$ on $(\Sigma\times\R_+, d(s\lambda_0))$
such that:
\begin{itemize}
\item{} $u (s) = 0$ outside $(R_0+\delta_1, R_1 + \delta_1)$,
\item{} $u (R_1) = 1$,
\item{} $u$ is non-decreasing on $[R_0,R_1]$ and its derivative there satisfies $u'(s) \leq (R_1 -
R_0)^{-1} + \delta_2$,
\end{itemize}
Here $\delta_1,\delta_2$ are small positive constants. The obvious relation between the Hamiltonian flow
of $u$ and the Reeb flow of $\lambda_0$ easily implies that the
flow of $u$ does not have
chords from $\cH''$ to $\cL''$ of time-length $\leq (R_1 - R_0)^{-1} T  -\epsilon$,
where $\epsilon\to 0$ as $\delta_1,\delta_2\to 0$.
Note also that $u$ $1$-separates $\cF''$ and $\cC''$.
Therefore Theorem~\ref{thm-pb-new} implies that
\[
pb_4^{M,+} (\cH'',\cL'',\cF'',\cC'')\leq \bigg( (R_1-R_0)T \bigg)^{-1},
\]
which, by \eqref{eqn-pb4+-anti-symmetry}, yields \eqref{eqn-upper-bound-on-pb4+-of-a-tetragon-in-sympln} and
hence \eqref{eqn-upper-bound-on-pb4+-of-a-tetragon}. Let us emphasize that the above argument involves Hamiltonian chords connecting the walls of the Lagrangian tetragon, even though the actual applications
deal with the chords connecting its floor and ceiling.

\medskip

2) Let us now prove
\begin{equation}
\label{eqn-lower-bound-on-pb4+-of-a-tetragon}
pb_4^{M,+} ( \cF, \cC, \cL, \cH)\geq \bigg( (R_1-R_0)T \bigg)^{-1}.
\end{equation}

The sets
$\Lambda_\varepsilon$ $C^0$-converge to $\Lambda$ as $\varepsilon\to
0$. Moreover, we can subdivide $\Lambda_\varepsilon$ into four sets
$\cF_\varepsilon, \cC_\varepsilon, \cL_\varepsilon, \cH_\varepsilon$
that $C^0$-converge, respectively, to $\cF,\cC,\cL,\cH$ as
$\varepsilon\to 0$. In view of \eqref{eqn-Hausdorff-convergence} we
need to prove that
\[
pb_4^{M,+} (\cF_\varepsilon, \cC_\varepsilon, \cL_\varepsilon,
\cH_\varepsilon ) \geq \bigg( (R_1-R_0)T - \delta (\varepsilon)
\bigg)^{-1}
\]
for some $\delta (\varepsilon) \to 0$ as $\varepsilon\to 0$.

Denote by $D_{\gamma_\varepsilon}$ the disk bounded by $\gamma_\varepsilon$ in
$[R_0,R_1]\times [0,T]$. Set
\[
\Gamma_\varepsilon := \Phi (\{ \textrm{point}\}\times \gamma_\varepsilon)\subset \Lambda''_\varepsilon
\]
and let
\[
D_{\Gamma_\varepsilon} := \Phi (\{ \textrm{point}\}\times D_{\gamma_\varepsilon})
\]
be the disk bounded by $\Gamma_\varepsilon$.

Pick arbitrary smooth functions $F, G\in C^\infty_c (M)$
satisfying the following conditions:
\begin{equation}
\label{eqn-F-G-conditions}
F|_{\op (\cF_\varepsilon )} = G|_{\op (\cL_\varepsilon )} = 0,\  F|_{\op (\cC_\varepsilon )} = G|_{\op (\cH_\varepsilon )} = 1.
\end{equation}
(Here $\op$ stands for {\it some} open
neighborhood of a set).
Note that, by \eqref{eqn-F-G-conditions},
\begin{equation}
\label{eqn-dF-wedge-dG-vanishes-on-op-Lambda-varepsilon}
dF\wedge dG|_{\op (\Lambda_\varepsilon)}\equiv 0.
\end{equation}
and thus $FdG|_{\Lambda_\varepsilon}$ is a closed form, which verifies condition 2 of Proposition~\ref{prop-FdG-Lambda}.

We claim that $[FdG|_{\Lambda_\varepsilon}] = -c[\lambda|_{\Lambda_\varepsilon}]$ for some $c>0$.
In order to check it, we need to show
that the integrals of $\lambda$ and $FdG$ over any $a\in H_1 (\Lambda_\varepsilon ;\R)$
are proportional with the proportionality coefficient $-c$ for some $c>0$ independent of
$a$. Note that
\[
H_1 (\Lambda_\varepsilon ;\R) \cong H_1 (L;\R)\oplus \R\oplus H_1 (K;\R),
\]
where the isomorphism is given by $\nu_*$ and the
middle summand $\R$ is generated by the homology class of the curve $\nu (\Gamma_\varepsilon)$.
Thus, it is enough to check the proportionality for $a=\nu_* b$,
where $b\in H_1 (L;\R)\oplus \R\oplus H_1 (K;\R)$ is one of the following types:

\smallskip
\noindent
1. $b\in H_1 (L;\R)$.

\smallskip
\noindent
2. $b = [\Gamma_\varepsilon]$.

\smallskip
\noindent
3. $b \in H_1 (K;\R)$.

Let us consider the first case. Note that the form $\eta$ vanishes on the zero section $K \subset T^*K$.
Thus, by \eqref{eqn-nu-respects-primitives},
\begin{equation}
\label{eqn-nu-lambda-s-lambda0-on-Lambda-varepsilon}
(\nu^* \lambda)|_{\Lambda_\varepsilon} = (s\lambda_0)|_{\Lambda_\varepsilon}\;,
\end{equation}
and hence
\[
\lambda (\nu_* b)  = (s\lambda_0) (b) = 0,
\]
since $b\in H_1 (L;\R)$ and the contact form $\lambda_0$ vanishes on the Legendrian submanifold $L$.
Since $FdG$ is zero on a neighborhood of $\cH$ and $b$ can be represented by a chain in $\cH$,
we also get $[FdG](b)=0$.

Let us move to the second case: $b = [\Gamma_\varepsilon]$.
A direct computation using \eqref{eqn-nu-lambda-s-lambda0-on-Lambda-varepsilon} and
the Stokes theorem shows that
\[
\lambda (\nu_* [\Gamma_\varepsilon])  = (s\lambda_0) (b) =
\int_{D_{\Gamma_\varepsilon}} d(s\lambda_0)  = (R_1 -R_0)T - \delta (\varepsilon) =:c >0,
\]
where the latter number $c$ is the Euclidean area of $D_{\gamma_\varepsilon}$.
At the same time another direct calculation using \eqref{eqn-F-G-conditions} yields
\[
\int_{\Gamma_\varepsilon} \nu^* (FdG) =  -1.
\]
Thus, in this case
\[
FdG (\nu_* b) = -\lambda (\nu_* b) /c.
\]

Finally, consider the third case: $b \in H_1 (K;\R)$.
By \eqref{eqn-nu-lambda-s-lambda0-on-Lambda-varepsilon} and since $s\lambda_0$ vanishes on $K$,
\[
\lambda (\nu_* b) = (s\lambda_0) (b) =0.
\]
Moreover, $\nu_* b$ can be represented by a curve that lies in $\cH$. Since, by \eqref{eqn-F-G-conditions},
$FdG$ is zero on a neighborhood of $\cH$, we get that
\[
FdG (\nu_* b) = 0.
\]

Thus, $[FdG|_{\Lambda_\varepsilon}] = -c[\lambda|_{\Lambda_\varepsilon}]$ for $c>0$.
This verifies the condition 3 of Proposition~\ref{prop-FdG-Lambda}.
Condition 1 is a part of the hypothesis of the theorem that we are proving.
Therefore Proposition~\ref{prop-FdG-Lambda} can be applied and we get
\[
\max_M \{F,G\}\geq 1/c=\bigg( (R_1-R_0)T - \delta (\varepsilon)\bigg)^{-1}.
\]
Since this is true for any $F,G\in C^\infty_c (M)$ satisfying \eqref{eqn-F-G-conditions}, we get
\[
pb_4^{M,+} (\cF_\varepsilon, \cC_\varepsilon, \cL_\varepsilon, \cH_\varepsilon) \geq \bigg( (R_1 -R_0)T
- \delta(\varepsilon)\bigg)^{-1},
\]
where $\delta (\varepsilon) \to 0$ as $\varepsilon\to 0$, as required.

This proves \eqref{eqn-lower-bound-on-pb4+-of-a-tetragon}. Together with \eqref{eqn-upper-bound-on-pb4+-of-a-tetragon} proved above
this implies the
proposition.
\Qed

\begin{rem}
\label{rem-generalization-to-weakly-non-exact-case} {\rm Recall that
a symplectic manifold $(M,\omega)$ is called {\it symplectically
aspherical} if $[\omega]$ vanishes on all homology classes in
$H_2(M,\R)$ represented by spheres, and a Lagrangian submanifold
$Z\subset (M,\omega)$ is called {\it weakly exact} if $[\omega]$
vanishes on the subspace $H_2^D (M,Z;\R) \subset H_2(M,Z;\R)$
generated by discs whose boundary lies on $Z$.

Theorem~\ref{thm-pb4-Mohnke} admits a rather straightforward
generalization to the case when $(M,\omega)$ is a symplectically
aspherical symplectic manifold and the Lagrangian isotopy class of
$\Lambda_\varepsilon$ does not contain weakly exact Lagrangian
submanifolds.

Namely, for such $(M,\omega)$ and $\Lambda$ assume in addition that
$$H_2^D (M,\Lambda;\R) = \im \nu_\sharp + Q,$$
where
\[
\nu_\sharp: H_2^D (U\times \cI\times W, \Lambda';\R)\to H_2^D (M,\Lambda;\R)
\]
is induced by $\nu$ and $Q\subset H_2^D (M,\Lambda;\R)$ is a
subgroup such that $\omega (Q)=0$ and $\partial (Q)$ lies in the
image of the natural inclusion $H_1 (L;\R)\to H_1 (\Lambda;\R)$.
Under these assumptions one can deduce the same conclusion as in
Theorem~\ref{thm-pb4-Mohnke}. }
\end{rem}

\subsection{Non-exactness of Lagrangians in Liouville manifolds}
\label{subsec-no-exact-Lagr-submfds-pfs}

In this section we describe classes of  Lagrangian tetragons in
exact symplectic manifolds $(M,d\lambda)$ whose Lagrangian isotopy classes contain no
exact Lagrangian submanifolds. Applying Corollary~\ref{cor-main} to such tetragons, we
shall prove our main results on the interlinking stated in the introduction.

Recall that an exact symplectic manifold $(M,\omega=d\lambda)$, with a fixed primitive $\lambda$ of the symplectic form $\omega$, is called {\it Liouville} if the
following two conditions hold:
\begin{itemize}
\item the flow $\theta_t$ of the Liouville vector field $v$ given by $i_v\omega = \lambda$ is complete;
 \item there exists
a closed connected
hypersurface $\Sigma$ in $M$ transversal to $v$ bounding a domain $U \subset M$ with
compact closure such that $M = U \sqcup \bigcup_{t \geq 0} \theta_t(\Sigma)$.
\end{itemize}
In this situation
$\lambda_0:=\lambda|_\Sigma$ is a contact $1$-form on $\Sigma$ defining a contact structure $\xi = \ker \lambda$. The set $\bigcup_{t\in \R} \theta_t(\Sigma) \subset M$ can be canonically symplectically identified with the symplectization $\Sigma\times \R_+$ of $\Sigma$, so that $\lambda$ is identified with the form $s\lambda_0$ on $\Sigma\times \R_+$.
We call the contact manifold $(\Sigma,\xi)$ {\it the ideal contact boundary} of the Liouville manifold $M$. One readily checks that the ideal contact boundary is uniquely defined up to a contactomorphism.

Examples of Liouvile manifolds include, for instance, the cotangent bundle $T^*V$ of a closed manifold $V$
equipped with the canonical Liouville form and the linear vector space $\R^{2n}(p,q)$ with the Liouville
form $(pdq-qdp)/2$. Here the ideal contact boundaries are the unit sphere bundle of $V$
with respect to any Riemannian metric, and the standard contact sphere, respectively. The product of two Liouville manifolds is again Liouville.

In what follows we focus on the case when $(\Sigma,\xi)$ is the ideal contact boundary of a Liouville manifold $M$, and fix the notation $\theta_t$ for the Liouville flow on $M$.

\medskip
\noindent
\begin{thm}
\label{thm-interlnkg-in-symplectn-for-bdry-of-Liouv-domain}
The exact symplectic manifold $(\Sigma \times \R_+\times T^*\T^n, d(s\lambda_0 + pdq))$ admits no closed exact Lagrangian
submanifolds. In particular, the pair $(\Sigma, L)$ is stably non-exact for every choice
of a closed Legendrian submanifold $L \subset \Sigma$.
\end{thm}

\begin{proof}

Assume on the contrary that there exists a closed exact Lagrangian
submanifold $Z$ of $(\Sigma\times \R_+\times T^* \T^m,d(s\lambda_0 +
pdq))$. Let $Z_t$ be the image of $Z$ under the Liouville flow of $M \times T^*T^m$  given by
$(x,p,q) \mapsto (\theta_t(x),tp,q)$. Note that Lagrangian submanifolds $Z_t$ are exact for all $t$, and $Z_\tau$ is disjoint from $Z$ for sufficiently large $\tau$. This contradicts to a theorem by Gromov stating
that a symplectically aspherical and geometrically bounded at infinity symplectic manifold -- and, in particular, any Liouville manifold -- does not admit a displaceable closed exact Lagrangian submanifold
\cite{Gromov}, \cite{ALP}.
\end{proof}

\medskip
\noindent
\begin{rem}\label{rem-weakly-exact-filling}{\rm
By a theorem of Gromov \cite{Gromov} (cf. \cite{ALP}),
any closed weakly exact Lagrangian submanifold of a geometrically bounded symplectically aspherical symplectic manifold $(M,\omega)$ is non-displaceable. The geometric boundedness assumption, which holds, for instance for
Liouville manifolds and cotangent bundles, enables one to control the behavior of pseudo-holomorphic curves
in $M$ ``at infinity". A minor modification of the above arguments
yields the following claim:

Let $(\Sigma, \ker \lambda_0)$ be an arbitrary (not necessarily closed) contact manifold.
Assume that for some $a>0$ the domain $\Sigma \times (a,+\infty) \subset (\Sigma \times \R_+, d(s\lambda_0))$ admits a symplectic embedding $\nu$ to a geometrically bounded symplectic manifold $(M,\omega)$. Let
$L \subset \Sigma$ be a closed connected Legendrian submanifold so that the morphism $H_2^D (\Sigma, L)\to H_2^D (M,L)/\ker \omega$ induced by $\nu$ is onto. (Here $\ker \omega$ is the kernel of the symplectic area map $D\mapsto \int_D \omega$ on $H_2^D (M,L)$).
Then the pair $(\Sigma,L)$ is stably non-exact.

If $\omega$ is exact, i.e.
$\omega=d\lambda$, the assumption on the morphism of the relative homology groups can be removed, provided $\nu^*\lambda= s\lambda_0$.

It would be interesting to find meaningful dynamical applications of this result.
}
\end{rem}

We have just proved that any Lagrangian tetragon in the symplectization $(\Sigma\times \R_+, d(s\lambda_0))$ of $\Sigma$ is not Lagrangian isotopic to an exact Lagrangian submanifold in $(\Sigma\times \R_+, d(s\lambda_0))$.  {\it Can it be Lagrangian isotopic to an exact Lagrangian submanifold in $M$}?
This question is still open. Below we answer it in negative in two particular cases.

First, we consider the class of subcritical Weinstein manifolds. Recall that a Liouville manifold is {\it Weinstein} if the Liouville vector field $v$ is gradient-like for a proper bounded from below Morse function $f$ with finite number of critical points. It is known that the indices of all critical points of $f$ are $\leq n$,
where $\dim M = 2n$. A Weinstein manifold $M$ is called {\it subcritical} if for some choice of $f$ all the indices are $< n$, and {\it critical} otherwise.

\medskip
\noindent
\begin{exam}
\label{exam-S-2k-1} {\rm The standard symplectic vector space $\R^{2n}(p,q)$ equipped with the Liouville form $\lambda= (pdq-qdp)/2$ is Weinstein subcritical. Indeed, the Liouville vector field
is $v=(p\partial/\partial p + q\partial/\partial q)/2$. It is gradient like for the proper non-negative function
${f}(p,q)=p^2+q^2$ whose only critical point is non-degenerate minimum at $0$. The ideal contact
boundary of $\R^{2n}$ is the unit sphere $\SP^{2n-1} \subset \R^{2n}$ equipped with standard contact structure
$\xi$ defined by the restriction of $\lambda$.}
\end{exam}

\medskip
\noindent More generally, the product $M \times \R^2$ with any Liouville $M$ is subcritical. In contrast to that, the cotangent bundle $T^*V$ of a closed manifold $V$ is critical.

\medskip
\noindent
\begin{thm}\label{thm-interlinking-in-subcritical-Weinstein-mfd}
Assume that $(M, \omega=d\lambda)$ is subcritical. Then the exact
symplectic manifold $(M \times T^* \T^m, d(\lambda + pdq))$ admits no
exact Lagrangian submanifolds.
\end{thm}

\begin{proof} Biran and Cieliebak showed \cite{Biran-Ciel-IsrJMath}
that every compact subset of $M$ (and a fortiori of $M \times T^* \T^m$)
is displaceable. The existence of closed exact Lagrangian submanifolds would
again violate the theorem of Gromov  cited in the previous proof, yielding a
contradiction.
\end{proof}

\medskip
\noindent
{\bf Proof of Theorem~\ref{thm-inter-2}.} We work in the situation of Example~\ref{exam-S-2k-1}.
By Theorem~\ref{thm-interlinking-in-subcritical-Weinstein-mfd} and Corollary \ref{cor-main}, any
Lagrangian tetragon $\Lambda\subset (\R^{2n}, dp\wedge dq)$, built
using any closed connected Legendrian submanifold of
$(\SP^{2n-1},\xi)$ and some $R_0,R_1,T$, is stably
$(R_1-R_0)T$-interlinked in $(\R^{2n}, dp\wedge dq)$.
\Qed

\medskip
Next, let  $V$ be a closed manifold, and $\Sigma \subset T^*V$ be the unit sphere bundle
with respect to any Riemannian metric on $V$. Even though $T^*V$ is a critical Weinstein manifold,
we have the following result.

\medskip
\noindent
\begin{thm}
\label{thm-no-exact-Lagr-isotopic-to-Lambda-in-cotangent-bundles}
The symplectic manifold $T^*V \times T^* \T^m$ admits no
exact Lagrangian submanifolds Lagrangian isotopic to $\Lambda_\varepsilon$.
\end{thm}

\begin{proof}
Assume, by contradiction, that $Z \subset T^* V\times T^* \T^m$ is an
exact Lagrangian submanifold Lagrangian isotopic to
$\Lambda_\varepsilon$. Then the $\Z_2$-degree of the projections $\Lambda_\varepsilon\to V\times \T^m$ and $Z\to V\times \T^m$
are equal.
By Kragh's theorem \cite[Cor. 1.1]{Kragh},
the latter degree is non-zero. (Recall that Kragh's theorem says that the $\Z_2$-degree of the projection $X\to Y$ of {\it any} closed exact Lagrangian submanifold $X$ of the cotangent bundle of {\it any} manifold $Y$ is non-zero).
Thus, the $\Z_2$-degree of the projection $\Lambda_\varepsilon\to V\times \T^m$ is non-zero.

On the other hand, note that the Lagrangian isotopy type of
$\Lambda_\varepsilon$ in $T^* V \times T^* \T^m$ is independent of
the parameters $R_0,R_1,T$ (where $\varepsilon$ is always chosen
sufficiently small depending on $R_0,R_1,T$). Thus, without loss of
generality, we may assume that $T$ is sufficiently small (depending
on $R_0$ and $R_1$) so that $\pi (\Lambda''_\varepsilon)$ lies in a
small neighborhood of $\pi (L)$ in $V$. Note that $\pi(L)\subsetneq
V$, since $\dim L = \dim V -1$. Therefore the projection $\pi:
\Lambda''_\varepsilon\to V$ is not surjective. Hence the projection
$\pi\times Id: \Lambda_\varepsilon\to V\times \T^m$ also is not
surjective and therefore the $\Z_2$-degree of the projection
$\Lambda_\varepsilon\to V\times \T^m$ is zero. We have got a
contradiction. This finishes the proof. \end{proof}

\medskip
\noindent
\begin{rem}
\label{rem-noncompact-Sigma}
{\rm
In fact, the same argument works verbatim for tetragons in $T^*V$ with non-compact $V$.

As we noted above, Theorem~\ref{thm-interlnkg-in-symplectn-for-bdry-of-Liouv-domain} implies that
if $\Sigma$ is the unit cotangent bundle (with respect to an arbitrary Riemannian metric)
of a {\it closed} manifold $V$, then for any closed connected Legendrian submanifold $L$ of $\Sigma$,
the pair $(\Sigma, L)$ is non-exact. The version of Theorem~\ref{thm-no-exact-Lagr-isotopic-to-Lambda-in-cotangent-bundles} for an arbitrary $V$
immediately yields the non-exactness of $(\Sigma, L)$ for any $V$.
}
\end{rem}

\medskip
\noindent
{\bf Proof of Theorem~\ref{thm-inter-1}.}
The Lagrangian tetragon $\Lambda$ in $T^* \T^k$ appearing in the statement of the theorem
is exactly the one constructed as above in the case where $\Sigma$ is
the unit sphere bundle with respect to the Euclidean metric on $V:= \T^k$, $L$ is a Legendrian fiber
of $\Sigma$, $K=\{\textrm{a point}\}$, and $0<R_0<R_1$, $0<T<1/2$ are some numbers.
Thus, Theorem~\ref{thm-no-exact-Lagr-isotopic-to-Lambda-in-cotangent-bundles} and Corollary \ref{cor-main}
can be applied, yielding that $\Lambda$ is stably $(R_1-R_0)T$-interlinked.
\Qed

\noindent
\begin{rem}
\label{rem-the-ineq-for-pb4+-is-exact-for-cot-bundle-of-1-dim-mfd}
{\rm
In view of Theorem~\ref{thm-pb-new} this means that $(R_1-R_0)T$ is the smallest
$\kappa$ for which the Lagrangian tetragon $\Lambda$ appearing in Theorem~\ref{thm-inter-1}
is $\kappa$-interlinked.
}
\end{rem}

\bigskip
\noindent
{\bf Acknowledgments.} We thank Vadim Kaloshin for very useful discussions and stimulating questions on Hamiltonian instabilities and Albert Fathi for communicating to us Theorem~\ref{thm-fathi}. We are grateful to Paul Biran and Tobias Ekholm for useful comments on Reeb dynamics and to an anonymous referee for corrections.
L.P. thanks the University of Chicago,
where parts of this paper were written, for an excellent research atmosphere.

\bibliographystyle{alpha}

\begin{thebibliography}{99}

\bibitem{Alves} Alves, M.R.R., \emph{Legendrian contact homology and topological entropy},
preprint arXiv:1410.3381, 2014.

\bibitem{Arnold-MathMethods} Arnold, V., \emph{Mathematical methods of classical mechanics}. Springer-Verlag, New York, 1989.

\bibitem{ALP} Audin, M., Lalonde, F., Polterovich, L.,
\emph{Symplectic rigidity: Lagrangian submanifolds}, in \emph{Holomorphic curves in symplectic geometry, 271-321},
Progr. Math., \textbf{117}, Birkh\"auser, Basel, 1994.

\bibitem{BEP} Buhovsky, L., Entov, M., Polterovich, L., \emph{Poisson brackets and symplectic invariants},
Selecta Math.  \textbf{18} (2012), 89-157.

\bibitem{Biran-Ciel-IsrJMath}  Biran, P., Cieliebak, K., \emph{Lagrangian embeddings into subcritical Stein manifolds}, Israel J. Math. \textbf{127} (2002), 221-244.

\bibitem{Bounemoura-Kaloshin} Bounemoura, A., Kaloshin, V., \emph{Generic diffusion for a class of non-convex Hamiltonians with two degrees of
freedom}, Moscow Math. J., \textbf{14} (2014), 181-203.


\bibitem{Ciel-Eliash} Cieliebak, K., Eliashberg, Y.,
\emph{From Stein to Weinstein and back.
Symplectic geometry of affine complex manifolds}. AMS, Providence, RI, 2012.



\bibitem{Ekholm} Ekholm, T., \emph{Rational SFT, linearized Legendrian contact homology, and Lagrangian Floer cohomology}, in \emph{``Perspectives in analysis, geometry, and topology", 109-145}, Birkh\"{a}user, Boston, 2012.



\bibitem{EG} Eliashberg, Y., Gromov, M., \emph{Convex symplectic manifolds}, in \emph{Several complex variables and complex geometry, Part 2 (Santa Cruz, CA, 1989), 135-162},
Proc. Sympos. Pure Math., \textbf{52}, Part 2, AMS, Providence, RI, 1991.

\bibitem{Eliash2004} Eliashberg, Y., \emph{A few remarks about symplectic filling}, Geom. Topol. \textbf{8} (2004), 277-293.





\bibitem{Etnyre2004} Etnyre, J., \emph{On symplectic fillings}, Algebr. Geom. Topol. \textbf{4} (2004), 73–80.

\bibitem{Fathi} Fathi, A.,  \emph{An Urysohn-type theorem under a dynamical
constraint}, J. of Modern Dynamics \textbf{10} (2015), 331-338.

\bibitem{Floer} Floer A., \emph{Morse theory for Lagrangian intersections},  Journal of Differential Geometry \textbf{28} (1988), 513-547.



\bibitem{Gromov} Gromov, M., \emph{Pseudoholomorphic curves in
symplectic manifolds}, Invent. Math. \textbf{82} (1985), 307-347.

\bibitem{Hagedorn} Hagedorn, P.,
\emph{Die Umkehrung der Stabilit\"{a}tss\"{a}tze von
Lagrange-Di\-rich\-let und Routh}, Arch. Rational Mech. Anal.
\textbf{42} (1971), 281-316.

\bibitem{Hofer} Hofer, H., \emph{Lagrangian embeddings and critical point theory},
Ann. Inst. H. Poincar\'e Anal. Non Lin\'eaire \textbf{2} (1985), 407–462.

\bibitem{Kragh} Kragh, T.,
\emph{Parametrized ring-spectra and the nearby Lagrangian conjecture
(with an appendix by Mohammed Abouzaid)},  Geom. Topol. \textbf{17}
(2013), 639-731.



\bibitem{Laud-Sik} Laudenbach, F., Sikorav, J.-C.,
\emph{Persistance d'intersection avec la section nulle au cours d'une isotopie hamiltonienne dans un fibr\'e cotangent},
Invent. Math. \textbf{82} (1985), 349–357.


\bibitem{Lisi} Lisi, S., \emph{Homoclinic orbits and Lagrangian embeddings},
Int. Math. Res. Not. IMRN (2008), no. 5, Art. ID rnm 151, 16 pp.

\bibitem{Mohnke} Mohnke, K., \emph{Holomorphic disks and the chord conjecture}, Ann. of Math. \textbf{154}
(2001), 219-222.

\bibitem{Moser} Moser, J., \emph{On the volume elements on a
manifold}, Trans. AMS \textbf{120} (1965), 288-294.

\bibitem{Murphy-exact-Lagr} Murphy, E., \emph{Closed exact Lagrangians in the symplectization of contact manifolds}, preprint, arXiv:1304.6620, 2013.

\bibitem{Nekhoroshev-UMN}  Nehoro\v{s}hev, N.N., \emph{An exponential estimate
of the time of stability of nearly integrable Hamiltonian systems}, Russian Math. Surveys \textbf{32} (1977), 1-65.

\bibitem{Taliaferro} Taliaferro, S.,
\emph{An inversion of the Lagrange-Dirichlet stability theorem},
Arch. Rational Mech. Anal. \textbf{73} (1980), 183-190.


\end{thebibliography}

\end{document}